\documentclass[oneside]{amsart}
\usepackage[utf8]{inputenc}
\usepackage{amstext}
\usepackage{amsthm}
\usepackage{amssymb}
\usepackage{graphicx}
\usepackage{esint}
\usepackage{amsmath}
\usepackage{mathtools}

\makeatletter
\numberwithin{equation}{section}
\numberwithin{figure}{section}
\theoremstyle{plain}
\newtheorem{theorem}{\protect\theoremname}
\theoremstyle{plain}
\newtheorem{lemma}[theorem]{\protect\lemmaname}

\newcommand{\R}{\mathbb{R}}
\newcommand{\C}{\mathbb{C}}
\newcommand{\Poly}[1]{P_{#1}^{(\alpha)}}
\DeclarePairedDelimiter{\floor}{\lfloor}{\rfloor}
\DeclareMathOperator{\sign}{sign}

\global\long\def\Im{\operatorname{Im}}

\global\long\def\Arg{\operatorname{Arg}}

\makeatother

\providecommand{\lemmaname}{Lemma}
\providecommand{\theoremname}{Theorem}

\begin{document}

\title[Power generating function]{Zero Distribution of Polynomials Generated by a Power of a Cubic Polynomial}
\author{Travis Steele \and Khang Tran}
\email{khangt@mail.fresnostate.edu \and t\_steele@mail.fresnostate.edu}
\address{Department of Mathematics, California State University, Fresno.\\
5245 North Backer Avenue M/S PB108 Fresno, CA 93740}
\begin{abstract}
For each $\alpha>0$ and $A(z),B(z)\in\mathbb{C}[z]$,
we study the zero distribution of the sequence of polynomials $\left\{ P_{m}^{(\alpha)}(z)\right\} _{m=0}^{\infty}$
generated by $(1+B(z)t+A(z)t^{3})^{-\alpha}$. We show that for large
$m$, the zeros of $P_{m}^{(\alpha)}(z)$ lie on an explicit curve
on the complex plane.
\end{abstract}

\maketitle

\section{Introduction}
The sequence Chebyshev polynomials of the second kind $\left\{ U_{m}(z)\right\} _{m=0}^{\infty}$
generated by 
\[
\sum_{m=0}^{\infty}U_{m}(z)t^{m}=\frac{1}{1-2zt+t^{2}}
\]
plays an important role in mathematics as it forms an orthogonal basis
with respect to the weight function $1/\sqrt{1-x^{2}}$ supported
on the interval $(-1,1)$ \cite[page 101]{aar}. As a consequence
of this orthogonality, its zeros lie on this $(-1,1)$ interval.
The orthogonality of the sequence and reality of the zeros also hold
for the sequence of Gegenbauer polynomials $\left\{ C_{m}^{(\alpha)}(z)\right\} _{m=0}^{\infty}$
, $\alpha>-1/2$, \cite[page 302]{aar} generated by 
\begin{equation}
\sum_{m=0}^{\infty}C_{m}^{(\alpha)}(z)t^{m}=\frac{1}{\left(1-2zt+t^{2}\right)^{\alpha}}.\label{eq:Gegenbauergen}
\end{equation}
The reality of theses zeros implies that the zeros of a more general
class of polynomials $\left\{ Q_{m}^{(\alpha)}(z)\right\} _{m=0}^{\infty}$
generated by 
\begin{equation}
\sum_{m=0}^{\infty}Q_{m}^{(\alpha)}(z)t^{m}=\frac{1}{(1+B(z)t+A(z)t^{2})^{\alpha}},\qquad A(z),B(z)\in\mathbb{C}[z],\label{eq:Q_mgen}
\end{equation}
lie on the curve 
\[
\Im\frac{B^{2}(z)}{A(z)}=0\qquad\text{and}\qquad0<\Re\frac{B^{2}(z)}{A(z)}<4.
\]
In the case $B(z)=-2z$ and $A(z)=1$, this curve becomes the interval
$(-1,1)$. To quickly see this implication, for each $z$ where $A(z)\notin\mathbb{R^{-}}$
we make the substitution $t\rightarrow t/\sqrt{A(z)}$ (with the principal
cut) in (\ref{eq:Q_mgen}) to conclude 
\[
\sum_{m=0}^{\infty}Q_{m}^{(\alpha)}(z)\frac{t^{m}}{A(z)^{m/2}}=\frac{1}{(1+B(z)t/\sqrt{A(z)}+t^{2})^{\alpha}}.
\]
We compare the right side with (\ref{eq:Gegenbauergen}) to arrive
at 
\[
A(z)^{m/2}C_{m}^{(\alpha)}\left(-\frac{B(z)}{2\sqrt{A(z)}}\right)=Q_{m}^{(\alpha)}(z).
\]
Since Gegenbauer polynomials are odd/even if $m$ is odd/even, the
singularities of the left side are removable and thus the two sides
are equal for all $z$ by analytic continuation. The claim that the
zeros of $Q_{m}^{(\alpha)}(z)$ lie on the given curve follows from
the fact that the zeros of Gegenbauer polynomials lie on $(-1,1)$.

Motivated by the generating function of $Q_{m}^{(\alpha)}(z)$ in
(\ref{eq:Q_mgen}), in this paper, we study the sequence of polynomial
$\left\{ H_{m}^{(\alpha)}(z)\right\} _{m=0}^{\infty}$ generated by
$1/(1+B(z)t+A(z)t^{3})^{\alpha}$. We show that for $\alpha>0$ and
for all large $m$, the zeros of $Q_{m}^{(\alpha)}(z)$ lie on an
explicit curve given in the theorem below.

\begin{theorem}For $\alpha>0$ and $A(z),B(z)\in\mathbb{C}[z]$,
let $\left\{ H_{m}^{(\alpha)}(z)\right\} _{m=0}^{\infty}$ be the
sequence of polynomials generated by 
\begin{equation}
\sum_{m=0}^{\infty}H_{m}^{(\alpha)}(z)t^{m}=\frac{1}{(1+B(z)t+A(z)t^{3})^{\alpha}}.\label{eq:generalgen}
\end{equation}
Then for all $m$ sufficiently large, the zeros of $H_{m}^{(\alpha)}(z)$ lie on
the curve defined by 
\[
\Im\frac{B^{3}(z)}{A(z)}=0\qquad\text{and }\qquad0<-\Re\frac{B(z)^{3}}{A(z)}<\frac{27}{4}.
\]
\end{theorem}

From the substitutions $t\rightarrow t/B(z)$, it suffices to show
that for any $\alpha>0$ the zeros of the zeros of $P_{m}^{(\alpha)}(z)$
generated by 
\begin{equation}
\sum_{m=0}^{\infty}P_{m}^{(\alpha)}(z)t^{m}=\frac{1}{(1+t+zt^{3})^{\alpha}}\label{eq:Pmgen}
\end{equation}
lie on the interval $(-\infty,-4/27)$ for all large $m$. The case
$\alpha=1$ recovers a sequence of polynomials generated by a rational
function, and the results in several articles (e.g. \cite{tran},
\cite{tran2018zeros}, \cite{forgacs2018hyperbolic}, and \cite{forgacs2016zeros})
imply that the zeros of $P_{m}^{(1)}(z)$ lie on this interval. Although
this claim holds for any $\alpha>0$, we claim that it suffices to
prove the claim for $\alpha\in(0,1)$. Indeed, we differentiate both
sides of 
\[
\frac{1}{(1+t+zt^{3})^{\alpha}}=\sum_{m=0}^{\infty}P_{m}^{(\alpha)}(z)t^{m}
\]
with respect to $z$ and obtain 
\begin{align*}
\sum_{m=0}^{\infty}\frac{d}{dz}\left(P_{m}^{(\alpha)}(z)\right)t^{m} & =\frac{-\alpha t^{3}}{(1+t+zt^{3})^{\alpha+1}}\\
 & =-\alpha\sum_{m=0}^{\infty}P_{m}^{(\alpha+1)}(z)t^{m+3}.
\end{align*}
We deduce that 
\[
-\alpha P_{m}^{(\alpha+1)}(z)=\frac{d}{dz}\left(P_{m+3}^{(\alpha)}(z)\right).
\]
We recall that if all the zeros of a polynomial lie on a real interval,
then so are those of its derivative. The identity above implies that if
we can show that the zeros of $P_{m}^{(\alpha)}(z)$ lie on $(-\infty,-4/27)$
for $\alpha\in(0,1)$, then this property also holds for all $\alpha>0$.
In fact, we will prove the theorem below and show how the zeros of $P_{m}^{(\alpha)}(z)$
distribute on this interval as $m\rightarrow\infty$.

\begin{theorem} \label{thm:zerodensity} Suppose $\alpha\in(0,1)$
and $P_{m}^{(\alpha)}(z)$ is generated by (\ref{eq:Pmgen}). For
large $m$, all the zeros of $P_{m}^{(\alpha)}(z)$ lie on $(-\infty,-4/27)$.
Moreover, the limiting probability density function of the zeros of
$P_{m}^{(\alpha)}(z)$ is 
\[
-\frac{3x\sqrt{x+1}}{2\pi z(3+2x)\sqrt{3-x}},\qquad z\in(-\infty,-4/27),
\]
where 
\begin{equation}
x=\frac{\sqrt[3]{\sqrt{3}\sqrt{27z^{4}+4z^{3}}-9z^{2}}}{\sqrt[3]{2}3^{2/3}z}-\frac{\sqrt[3]{\frac{2}{3}}}{\sqrt[3]{\sqrt{3}\sqrt{27z^{4}+4z^{3}}-9z^{2}}}\label{eq:xform}
\end{equation}
is the unique real zero in $t$ of $1+t+zt^{3}$ .\end{theorem}

For the plots of the all the zeros of $P_{m}^{(\alpha)}(z)$, $1\le m\le50$,
and the limiting probability density function, see Figures \ref{fig:zerosP_m}
and \ref{fig:limitingprob} respectively.

\begin{figure}
\centering \includegraphics[scale=0.5]{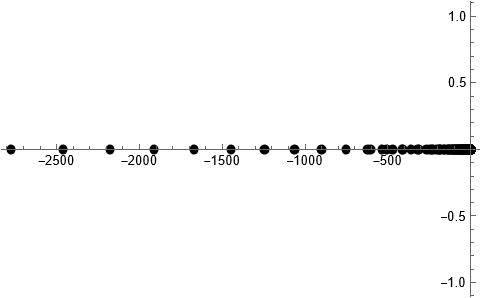} \caption{ Zeros of $P_{m}^{(\alpha)}(z)$ for $1\protect\leq m\protect\leq50$
and $\alpha=7.5$.}\label{fig:zerosP_m}
\end{figure}

\begin{figure}
\centering \includegraphics[scale=0.25]{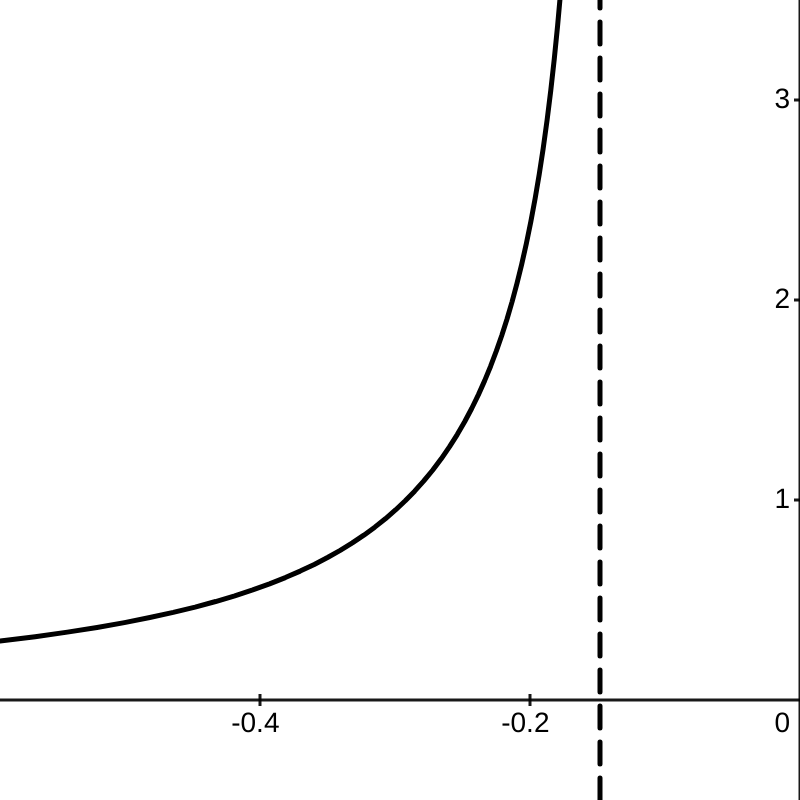} \caption{Limiting probability density function of the zeros of $\{P_{m}^{(\alpha)}(z)\}_{m=0}^{\infty}$.
The vertical dashed line is drawn at $z=-\frac{4}{27}$.}\label{fig:limitingprob}
\end{figure}

The remainder of this article is organized as follows. In Section
\ref{sec1} we establish some elementary results regarding $P_{m}^{(\alpha)}(z)$.
Section \ref{integral} is dedicated to expressing this polynomial
as the difference of two real integrals, whose behavior as $m\to\infty$
is studied in Section \ref{asymp}. Finally we prove Theorem \ref{thm:zerodensity}
in Section \ref{main}.

\section{Basic Properties}

\label{sec1} In this section we establish some basic results about
the polynomial $P_{m}^{(\alpha)}(z)$ defined in (\ref{eq:Pmgen}).
We begin by expressing this polynomial in its standard form.

\begin{lemma} \label{explicitForm} For each $\alpha>0$, let $\{P_{m}^{(\alpha)}(z)\}_{m=0}^{\infty}$
be the sequence of polynomials defined in (\ref{eq:Pmgen}). The following
holds for each $m$: 
\[
P_{m}^{(\alpha)}(z)=\sum_{k=0}^{\left\lfloor m/3\right\rfloor }\binom{-\alpha}{m-2k}\binom{m-2k}{k}z^{k}.
\]
In particular, 
\[
\deg P_{m}^{(\alpha)}(z)=\left\lfloor \frac{m}{3}\right\rfloor 
\]
and 
\[
\lim_{z\to-\infty}\sign P_{m}^{(\alpha)}(z)=(-1)^{m-\left\lfloor m/3\right\rfloor }.
\]
\end{lemma} 
\begin{proof}
We expand $(1+t+zt^{3})^{-\alpha}$ using the binomial series to obtain
\[
\frac{1}{(1+t+zt^{3})^{\alpha}}=\sum_{m=0}^{\infty}\binom{-\alpha}{m}(t+zt^{3})^{m}
\]
where 
\[
\binom{-\alpha}{m}=\frac{(-\alpha)(-\alpha-1)\cdots(-\alpha-m+1)}{m!}.
\]
The factor $(t+zt^{3})^{m}$ can be expanded using the binomial theorem:
\[
\frac{1}{(1+t+zt^{3})^{\alpha}}=\sum_{m=0}^{\infty}\binom{-\alpha}{m}\sum_{k=0}^{m}\binom{m}{k}t^{m-k}(zt^{3})^{k},
\]
which simplifies to 
\[
\frac{1}{(1+t+zt^{3})^{\alpha}}=\sum_{m=0}^{\infty}\sum_{k=0}^{m}\binom{-\alpha}{m}\binom{m}{k}z^{k}t^{m+2k}.
\]
The sum can be rearranged by collecting terms with the same power
of $t$: 
\[
\sum_{m=0}^{\infty}\sum_{k=0}^{m}\binom{-\alpha}{m}\binom{m}{k}z^{k}t^{m+2k}=\sum_{m=0}^{\infty}\sum_{p+2q=m}\binom{-\alpha}{p}\binom{p}{q}z^{q}t^{m}.
\]
Now parameterize the $1+\floor{m/2}$ non-negative integer solutions
to the equation $p+2q=m$ by $(p,q)=(m-2k,k)$ to obtain 
\[
\sum_{p+2q=m}\binom{-\alpha}{p}\binom{p}{q}z^{q}=\sum_{k=0}^{\floor{m/2}}\binom{-\alpha}{m-2k}\binom{m-2k}{k}z^{k}.
\]
Since $k$ is non-negative here, the factor $\binom{m-2k}{k}$ is
nonzero if and only if $k\leq m-2k$. This inequality implies that
the possible values of $k$ on the right side are $k=0,1,2,\dots,\left\lfloor m/3\right\rfloor $.
Thus, 
\[
\frac{1}{(1+t+zt^{3})^{\alpha}}=\sum_{m=0}^{\infty}\left(\sum_{k=0}^{\floor{m/3}}\binom{-\alpha}{m-2k}\binom{m-2k}{k}z^{k}\right)t^{m}.
\]
Finally, $\binom{-\alpha}{m-2k}$ is nonzero whenever $\alpha>0$
and $m-2k\geq0$. In fact, 
\begin{equation}
\sign\binom{-\alpha}{m-2k}=(-1)^{m}.\label{eq:signcoeff}
\end{equation}
The end behavior of $\Poly{m}(z)$ as $z\to-\infty$ is determined
by the leading coefficient, whose sign is eventually $(-1)^{m-\floor{m/3}}$. 
\end{proof}
As a consequence of Lemma \ref{explicitForm} and (\ref{eq:signcoeff}),
$\Poly{m}(z)\ne0$ when $z$ is non-negative. When $z$ is negative,
two cases arise in the behavior of the zeros of $1+t+zt^{3}$. As
a polynomial in $t$, the discriminant of $1+t+zt^{3}$ is $-27z^{2}-4z$,
which is non-negative if and only if $z\in[-\frac{4}{27},0]$. Hence,
$1+t+zt^{3}$ has one real root and a pair of conjugate non-real roots
whenever $z\in(-\infty,-\frac{4}{27})$. Our goal is to show that
all the roots of $\Poly{m}(z)$ lie in this interval for $m$ sufficiently
large. To achieve this goal, we will count the number of zeros of
$P_{m}^{(\alpha)}(z)$ on this interval and compare this number with
the degree of this polynomial. Thus, for the remainder of this paper,
we shall assume $z\in(-\infty,-\frac{4}{27})$ and denote by $x$
the real root of $1+t+zt^{3}$. A computer algebra system provides
an explicit formula of $x$ in (\ref{eq:xform}). We write the non-real
roots as $re^{\pm i\theta}$, where $r>0$ and $\theta\in(0,\pi)$.
Our next result lists some important relationships between these roots.

\begin{lemma} \label{rootprops} For each $z\in(-\infty,-\frac{4}{27})$
, let $x$ be the real root of $1+t+zt^{3}$ as a polynomial in $t$,
and $y=re^{i\theta}$, $r>0$, be its non-real root in the upper half-plane.
Then $\theta\in(2\pi/3,\pi)$ and 
\begin{align*}
r & =\frac{1-4\cos^{2}\theta}{2\cos\theta},\\
x & =4\cos^{2}\theta-1,\\
z & =\frac{4\cos^{2}\theta}{(1-4\cos^{2}\theta)^{3}},\\
|x-y|^{2} & =2x^{2}+r^{2}.
\end{align*}
\end{lemma} 
\begin{proof}
Vieta's formulas (see \cite{VF}) provide us the following relations
between the roots of $1+t+zt^{3}$: 
\begin{align*}
x+y+\overline{y} & =0,\\
xy+x\overline{y}+y\overline{y} & =\frac{1}{z},\\
xy\overline{y} & =-\frac{1}{z}.
\end{align*}
Write $x=qr$ for some $q\in\R$. Note that, since $z$ is negative,
$1+t+zt^{3}$ is negative for $t$ sufficiently large. But $1+0+z(0)^{3}=1$
is positive, so it follows by the Intermediate Value Theorem that
the real root $x$ of $1+t+zt^{3}$ is positive. Hence, $q$ is positive
as well. With this new notation, Vieta's formulas become 
\begin{align}
qr+re^{i\theta}+re^{-i\theta} & =0,\label{eq:sq}\\
qr^{2}e^{i\theta}+qr^{2}e^{-i\theta}+r^{2} & =\frac{1}{z},\text{ and}\label{eq:lin}\\
qr^{3} & =-\frac{1}{z}.\label{eq:const}
\end{align}
Now we solve for $q$ and $r$ in terms of $\theta$. From the first
of these equations, we divide by $r$ and then subtract $q$ on both
sides to obtain 
\begin{equation}
e^{-i\theta}+e^{-i\theta}=-q.\label{eq:q1}
\end{equation}
As $\cos\theta=(e^{-i\theta}+e^{-i\theta})/2$, this equation is equivalent
to 
\begin{equation}
q=-2\cos\theta.\label{eq:q2}
\end{equation}
Since we know $q$ is positive, this implies $\theta\in(\pi/2,\pi)$.
Next, we factor (\ref{eq:lin}) to obtain 
\[
r^{2}(q(e^{i\theta}+e^{-i\theta})+1)=\frac{1}{z}.
\]
Applying (\ref{eq:q1}), we deduce that 
\[
r^{2}(1-q^{2})=\frac{1}{z}.
\]
By (\ref{eq:q2}), this becomes 
\[
r^{2}(1-4\cos^{2}\theta)=\frac{1}{z}.
\]
We add (\ref{eq:const}) to this and we get 
\[
r^{3}q+r^{2}(1-4\cos^{2}\theta)=0.
\]
Dividing both sides by $r^{2}$ and then isolating $r$, we obtain
\[
r=\frac{-(1-4\cos^{2}\theta)}{q}.
\]
Finally, we apply (\ref{eq:q2}) again and conclude that 
\begin{equation}
r=\frac{1-4\cos^{2}\theta}{2\cos\theta}.\label{eq:r}
\end{equation}
Since the denominator of the right side is negative on the interval
$(\frac{\pi}{2},\pi)$, yet within this interval the numerator is
only negative when $\theta\in(\frac{2\pi}{3},\pi)$, and since $r$
is positive, it follows that $\theta\in(\frac{2\pi}{3},\pi)$. Consequently,
\begin{equation}
x=4\cos^{2}\theta-1.\label{eq:x}
\end{equation}
We can also compute $z$ in terms of $\theta$: 
\begin{equation}
z=-\frac{1}{qr^{3}}=\frac{4\cos^{2}\theta}{(1-4\cos^{2}\theta)^{3}}.\label{eq:z}
\end{equation}
To prove the last equation in this lemma, which is 
\begin{equation}
|x-y|^{2}=2x^{2}+r^{2},\label{eq:diff}
\end{equation}
we apply the Law of Cosines to the triangle in the complex plane formed
by the points $0$, $x$, and $y$, to conclude that 
\[
|x-y|^{2}=x^{2}+r^{2}-2xr\cos\theta,
\]
from which (\ref{eq:diff}) follows since $x=-2r\cos\theta$. 
\end{proof}

\section{A formula for $P_{m}^{(\alpha)}(z)$}

\label{integral}

Let $z\in(-\infty,-4/27)$ and $\alpha\in(0,1)$, and keep the notation
for the roots of $1+t+zt^{3}$ established in the previous section
(c.f. Lemma \ref{rootprops}). In this section we apply the Cauchy
Integral Formula (Lemma 10.2 of \cite{ST}) to $(1+t+zt^{3})^{-\alpha}$
to get an integral representation of each $P_{m}^{(\alpha)}(z)$.
Then we will deform the contour to express $P_{m}^{(\alpha)}(z)$
as the difference of two integrals. Finally, we analyze the asymptotic
behavior of each of these resulting integrals.

Since the zeros of $1+t+zt^{3}$ are $x$ and $re^{\pm i\theta}$,
we can factor $1+t+zt^{3}$ as 
\[
1+t+zt^{3}=-z(x-t)(t^{2}+xt+r^{2}),
\]
where the right side is the same as 
\[
-z(x-t)((t-r\cos\theta)^{2}+r^{2}\sin^{2}\theta).
\]
Define 
\begin{equation}
G_{m}(z,t)=\frac{1}{(-z)^{\alpha}(x-t)^{\alpha}(t^{2}+xt+r^{2})^{\alpha}t^{m+1}},\label{Gm}
\end{equation}
where $x$, $y$, $r$, and $\theta$ are as in Lemma \ref{rootprops}.
Since $-z$, $x-t$ and $t^{2}+xt+r^{2}$ are all positive real numbers
when $t=0$, using the principal cut, we have for small $t$: 
\[
\frac{1}{(1+t+zt^{3})^{\alpha}}=\frac{1}{(-z)^{\alpha}(x-t)^{\alpha}(t^{2}+xt+r^{2})^{\alpha}}.
\]
Applying the Cauchy Integral Formula here gives an expression for
each $P_{m}^{(\alpha)}(z)$ as a contour integral.

\begin{lemma} \label{int1} Let $x$, $y$, $r$, and $\theta$ be
as in Lemma \ref{rootprops} and $G_{m}$ be defined as in (\ref{Gm}).
There exists a small positive real number $\varrho$ such that for
all $m\geq0$, 
\[
P_{m}^{(\alpha)}(z)=\frac{1}{2\pi i}\oint_{|t|=\varrho}G_{m}(z,t)dt.
\]
\end{lemma}

The integral in this lemma is taken along a sufficiently small loop
around the origin. To deform this contour, we locate the singularities
of $G_{m}(z,t)$ (as a complex function of $t$).

\begin{lemma} Let $x$, $y$, $r$, and $\theta$ be as in Lemma
\ref{rootprops}. For any $\theta\in(2\pi/3,\pi)$, $t^{2}+xt+r^{2}$
is a negative real number if and only if $t=re^{\pm i\theta}\pm is$
for some positive real number $s$. \end{lemma} 
\begin{proof}
If $t=re^{\pm i\theta}\pm is$ for some positive $s\in\R$, then 
\begin{align*}
t^{2}-(2r\cos\theta)t+r^{2}= & r^{2}e^{\pm2i\theta}\pm2irse^{\pm i\theta}-s^{2}-(2r^{2}\cos\theta)e^{\pm i\theta}\mp2irs\cos\theta+r^{2}\\
= & (re^{\pm i\theta})^{2}-(2r^{2}\cos\theta)e^{\pm i\theta}+r^{2}-s^{2}\pm2irs(e^{\pm i\theta}-\cos\theta)
\end{align*}
where the sum of the first three terms in the last expression is 
\begin{align*}
(re^{\pm i\theta})^{2}-(2r^{2}\cos\theta)e^{\pm i\theta}+r^{2}= & r^{2}(\cos(2\theta)\pm i\sin(2\theta))\\
 & -2r^{2}(\cos^{2}\theta\pm i\cos\theta\sin\theta)+r^{2}\\
= & r^{2}(2\cos^{2}\theta-1\pm2i\sin\theta\cos\theta)\\
 & -2r^{2}(\cos^{2}\theta\pm i\cos\theta\sin\theta)+r^{2}\\
= & 0,
\end{align*}
and the sum of the last two terms is 
\begin{align*}
-s^{2}\pm2irs(e^{\pm i\theta}-\cos\theta) & =-s^{2}\pm2irs(\pm i\sin\theta)\\
 & =-s^{2}-2rs\sin\theta.
\end{align*}
Hence 
\[
t^{2}-(2r\cos\theta)t+r^{2}=-(s^{2}+2rs\sin\theta)<0
\]
since $\theta\in(\frac{\pi}{2},\pi)$.

Conversely, suppose $t^{2}-(2r\cos\theta)t+r^{2}$ is a negative real
number and $t=\rho e^{i\phi}$ for some $\rho>0$ and $\phi\in(-\pi,\pi]$.
If $\phi=\pi$ or $\phi=0$, then $t$ is real, from which we deduce
that $t^{2}-(2r\cos\theta)t+r^{2}$ is a positive real number (indeed,
$t^{2}-(2r\cos\theta)t+r^{2}$ attains its minimum value, $r^{2}(1-\cos\theta)$,
at $t=r\cos\theta$). So assume $\phi\in(-\pi,\pi)$ and $\phi\neq0$.
Then 
\[
0=\Im\left(t^{2}-(2r\cos\theta)t+r^{2}\right)=\rho^{2}\sin(2\phi)-\rho(2r\cos\theta)\sin(\phi),
\]
which implies
\[
2r\cos\theta=\frac{\rho\sin(2\phi)}{\sin\phi},
\]
wherefore 
\[
r\cos\theta=\rho\cos\phi,
\]
and hence 
\[
\Re\left(re^{i\theta}\right)=\Re\left(t\right).
\]
On the other hand, 
\[
\Re\left(t^{2}-(2r\cos\theta)t+r^{2}\right)<0,
\]
which implies 
\[
\rho^{2}\cos(2\phi)-\rho(2r\cos\theta)\cos(\phi)+r^{2}<0,
\]
whereby 
\[
\rho^{2}(2\cos^{2}\phi-1)-2(r\cos\theta)(\rho\cos\phi)+r^{2}<0,
\]
and thus 
\[
2\rho^{2}\cos^{2}\phi-2(\rho\cos\phi)(\rho\cos\phi)+r^{2}-\rho^{2}<0.
\]
This implies $r^{2}<\rho^{2}$ or equivalently $r<\rho$. Therefore,
$t^{2}-(2r\cos\theta)t+r^{2}$ is a negative real number if and only
if $t=re^{\pm i\theta}+is$ for some $s>0$. 
\end{proof}
By the previous result and the fact that $x-t$ is a negative real
number if and only if $t$ is a real number greater than $x$, we
obtain the following theorem. \begin{theorem} Let $G_{m}$ be defined
as in (\ref{Gm}), and let $x$, $y$, $r$, and $\theta$ be as in
Lemma \ref{rootprops}. Then with the principal cut, $G_{m}(z,t)$
is analytic on $\C\setminus(\{0\}\cup H\cup V_{+}\cup V_{-})$, where
\begin{align*}
H & :=\{x+t:t\geq0\}\\
V_{+} & :=\{y+it:t\geq0\}\\
V_{-} & :=\{\overline{y}-it:t\geq0\}.
\end{align*}
\end{theorem}

\begin{figure}
\centering \includegraphics[scale=0.25]{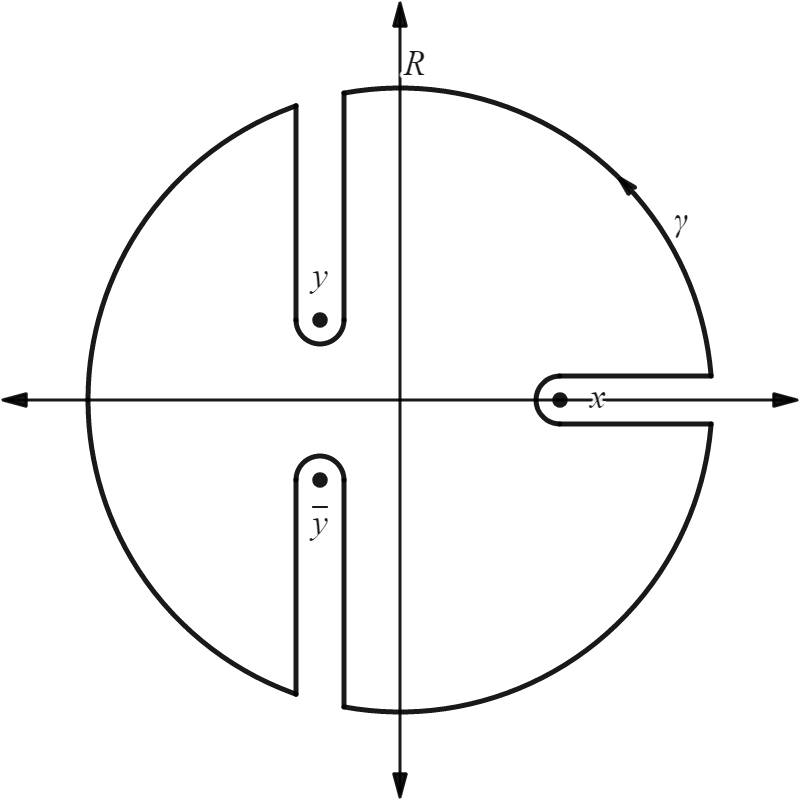} \caption{The contour $\gamma$}
\label{graph} 
\end{figure}

Now we deform our loop $|t|=\varrho$ into a new contour $\gamma$
by increasing its radius and hugging the cuts in the plane by some
small fixed margin. The contour $\gamma$ is plotted in Figure \ref{graph}.
More concretely, let $\varepsilon$ and $R$ be positive real numbers
with $\varrho<R$, and for any set of complex numbers $S$, let $d(\tau,S)$
denote the distance from $\tau$ to $S$, i.e., 
\[
d(\tau,S):=\inf_{s\in S}|\tau-s|.
\]
Define the set 
\[
D:=\{\tau\in\C:d(\tau,H\cup V_{+}\cup V_{-})=\varepsilon\}\cup\{\tau\in\C:|\tau|=R\}
\]
Then for small $\varepsilon$ the connected component of $\C\setminus D$
containing $0$ also contains the loop $|t|=\varrho$. The very same
connected component is simply connected, so its boundary $\gamma$
is homotopic to the loop $|t|=\varrho$. Thus, to compute the integral
in Lemma \ref{int1}, we can integrate along $\gamma$ instead of
the loop $|t|=\varrho$. The next result shows that as $R\to\infty$,
the portion of $\gamma$ coinciding with the circle of radius $R$
about the origin becomes negligible.

\begin{lemma} Let $G_{m}$ be defined as in (\ref{Gm}), and let
$x$, $y$, $r$, and $\theta$ be as in Lemma \ref{rootprops}. Define
$\sigma_{R}(t)=Re^{it}$ for $t\in[a,b]$, where $-\pi\leq a<b\leq\pi$
and $\theta,0,-\theta\not\in[a,b]$. Then 
\[
\lim_{R\to\infty}\int_{\sigma_{R}}G_{m}(z,t)dt=0.
\]
\end{lemma} 
\begin{proof}
We have 
\[
|G_{m}(z,t)|=\frac{1}{(-z)^{\alpha}|x-t|^{\alpha}|t^{2}+xt+r^{2}|^{\alpha}|t|^{m+1}}.
\]
For $|t|=R$ and $R$ sufficiently large we have 
\[
|t^{2}+xt+r^{2}|\geq R^{2}-xR-r^{2}
\]
and 
\[
|x-t|\geq R-x
\]
by the reverse triangle inequality. Consequently, 
\[
|G_{m}(z,t)|\leq\frac{1}{(-z)^{\alpha}(R-x)^{\alpha}(R^{2}-xR-r^{2})^{\alpha}R^{m+1}}.
\]
The length of $\sigma$ is no more than $2\pi R$, therefore 
\[
\int_{\sigma_{R}}|G_{m}(z,t)||dt|\leq\frac{2\pi R}{(-z)^{\alpha}|x-R|^{\alpha}R^{2\alpha}R^{m+1}}.
\]
Because 
\[
\lim_{R\to\infty}\frac{2\pi R}{(-z)^{\alpha}|x-R|^{\alpha}R^{2\alpha}R^{m+1}}=0,
\]
our desired conclusion follows. 
\end{proof}
Thus, as $R\rightarrow\infty$, we can disregard the larger circular
component of $\gamma$, and the parts of $\gamma$ coinciding with
line segments become rays. Similarly, the next result shows that we
can disregard the smaller semi-circular portions as well as $\varepsilon\rightarrow0$.

\begin{lemma} Let $\alpha\in(0,1)$ and $G_{m}$ be defined as in
(\ref{Gm}), and let $x$, $y$, $r$, and $\theta$ be as in Theorem
\ref{rootprops}. For small $\varepsilon>0$ define $\sigma_{1}:[\frac{\pi}{2},\frac{3\pi}{2}]\to\C$
by 
\[
\sigma_{1}(s)=\varepsilon e^{is}+x,
\]
and $\sigma_{2}:[\pi,2\pi]\to\C$ by 
\[
\sigma_{2}(s)=\varepsilon e^{is}+y,
\]
and define $\sigma_{3}:[0,\pi]\to\C$ by 
\[
\sigma_{3}(s)=\varepsilon e^{is}+\overline{y}.
\]
Then 
\[
\lim_{\varepsilon\to0}\int_{\sigma_{j}}G_{m}(z,t)dt=0
\]
for $j=1,2,3$. \end{lemma} 
\begin{proof}
Since 
\[
|\sigma_{1}(s)|\leq x-\varepsilon
\]
for all $s\in[\frac{\pi}{2},\frac{3\pi}{2}]$, we have 
\[
\frac{1}{|\sigma_{1}(s)|^{m+1}}\leq\frac{1}{(x-\varepsilon)^{m+1}}.
\]
As the function $t\mapsto t^{2}-(2r\cos\theta)t+r^{2}$ has no zeros
in the disc $|t-x|\leq\varepsilon$, there is a positive real number
$M$ such that for any $\varepsilon$ sufficiently small and $s\in[\frac{\pi}{2},\frac{3\pi}{2}]$
\[
\frac{1}{|(\sigma_{1}(s)^{2}-(2r\cos\theta)\sigma_{1}(s)+r^{2})^{\alpha}|}\leq M.
\]
Thus, for $\varepsilon$ sufficiently small, 
\[
\int_{\sigma_{1}}|G_{m}(z,t)||dt|\leq\pi\varepsilon\cdot\frac{M}{(-z)^{\alpha}\varepsilon^{\alpha}(x-\varepsilon)^{m+1}}=\frac{\pi M\varepsilon^{1-\alpha}}{(-z)^{\alpha}(x-\varepsilon)^{m+1}},
\]
from which we deduce that 
\[
\lim_{\varepsilon\to0}\int_{\sigma_{1}}G_{m}(z,t)dt=0
\]
for any $\alpha\in(0,1)$.

For the integral along $\sigma_{2}$, note that 
\[
\sigma_{2}(s)=r\cos\theta+ir\sin\theta+\varepsilon e^{is}.
\]
We apply the identity 
\[
t^{2}-(2r\cos\theta)t+r^{2}=(t-r\cos\theta)^{2}+(r\sin\theta)^{2},
\]
to conclude 
\begin{align*}
\sigma_{2}(s)^{2}-(2r\cos\theta)\sigma_{2}(s)+r^{2} & =(ir\sin\theta+\varepsilon e^{is})^{2}+(r\sin\theta)^{2}\\
 & =2i\varepsilon(\sin\theta)e^{is}+\varepsilon^{2}e^{2is}.
\end{align*}
By the reverse triangle inequality, for sufficiently small $\varepsilon$
\[
2\varepsilon s\sin\theta-\varepsilon^{2}\leq|2i\varepsilon(\sin\theta)e^{is}+\varepsilon^{2}e^{2is}|,
\]
and consequently 
\[
\frac{1}{|(\sigma_{2}(s)^{2}-(2r\cos\theta)\sigma_{2}(s)+r^{2})^{\alpha}|}\leq\frac{1}{\varepsilon^{\alpha}(2r\sin\theta-\varepsilon)^{\alpha}}
\]
for all $s\in[\pi,2\pi]$. The magnitude of any point on the image
of $\sigma_{2}$ exceeds $r-\varepsilon$, so 
\[
\frac{1}{|\sigma_{2}(s)^{m+1}|}\leq\frac{1}{(r-\varepsilon)^{m+1}}.
\]
As the function $t\mapsto x-t$ has no zeros in the disc $|t-y|\leq\varepsilon$,
there is a positive number $M_{2}$ such that for any $\varepsilon$
sufficiently small, 
\[
\frac{1}{|(x-\sigma_{2}(s))^{\alpha}|}\leq M_{2}.
\]
Thus, for $\varepsilon$ sufficiently small, 
\[
\begin{split}\int_{\sigma_{2}}|G_{m}(z,t)||dt| & \leq\pi\varepsilon\cdot\frac{M_{2}}{(-z)^{\alpha}\varepsilon^{\alpha}(2r\sin\theta-\varepsilon)^{\alpha}(r-\varepsilon)^{m+1}}\\
 & =\frac{\pi M_{2}\varepsilon^{1-\alpha}}{(-z)^{\alpha}(2r\sin\theta-\varepsilon)^{\alpha}(r-\varepsilon)^{m+1}}
\end{split}
\]
Therefore, for $\alpha\in(0,1)$, 
\[
\lim_{\varepsilon\to0}\int_{\sigma_{2}}G_{m}(z,t)dt=0.
\]

For the integral along $\sigma_{3}$, note that 
\[
\sigma_{3}(s)=\overline{\sigma_{2}(\pi-s)}.
\]
Thus, as a consequence of the previous equation, 
\[
\lim_{\varepsilon\to0}\int_{\sigma_{3}}G_{m}(z,t)dt=0
\]
whenever $\alpha\in(0,1)$. 
\end{proof}
As $R\rightarrow\infty$ and $\varepsilon\rightarrow0$, the six line
segments of $\gamma$ (see Figure \ref{graph}) become six rays. By
adding together the two integrals along the horizontal rays, we obtain
a real integral, as shown in the next lemma.

\begin{lemma} Let $x$, $y$, $r$, and $\theta$ be as in Lemma
\ref{rootprops}. For each small $\varepsilon>0$, define $\ell:[0,\infty)\to\C$
by 
\[
\ell(s)=x+s+i\varepsilon.
\]
Then $G_{m}(z,t)$, defined in \eqref{Gm}, satisfies 
\[
\begin{split} & \int_{\ell}G_{m}(z,t)dt-\int_{\overline{\ell}}G_{m}(z,t)dt\\
 & \to\int_{0}^{\infty}\frac{2i\sin(\alpha\pi)}{(-z)^{\alpha}t^{\alpha}(t^{2}+3xt+|x-y|^{2})^{\alpha}(x+t)^{m+1}}dt
\end{split}
\]
as $\varepsilon\to0$. \end{lemma} 
\begin{proof}
Since $x=-2r\cos\theta$ (by Lemma \ref{rootprops}), we have 
\begin{align}
\lim_{\varepsilon\to0}(\ell(s))^{2}+x\ell(s)+r^{2}) & =\lim_{\epsilon\rightarrow0}((\ell(s)-r\cos\theta)^{2}+(r\sin\theta)^{2})\nonumber \\
 & =(s-3r\cos\theta)^{2}+(r\sin\theta)^{2}\nonumber \\
 & =s^{2}-6(r\cos\theta)s+9(r\cos\theta)^{2}+(r\sin\theta)^{2}\nonumber \\
 & =s^{2}+3xs+8r^{2}\cos^{2}\theta+r^{2}\nonumber \\
 & =s^{2}+3xs+|x-y|^{2}.\label{eq:limitquadratic}
\end{align}
Moreover, 
\[
x-\ell(s)=-s-i\varepsilon
\]
yields 
\[
\lim_{\varepsilon\to0}(x-\ell(s))^{\alpha}=s^{\alpha}e^{-i\pi\alpha},
\]
from which we deduce that 
\[
\lim_{\varepsilon\to0}G_{m}(z,\ell(s))=\frac{e^{i\pi\alpha}}{(-z)^{\alpha}s^{\alpha}(s^{2}+3xs+|x-y|^{2})^{\alpha}(x+s)^{m+1}},
\]
and likewise 
\[
\lim_{\varepsilon\to0}G_{m}(z,\overline{\gamma(t)})=\frac{e^{-i\pi\alpha}}{(-z)^{\alpha}s^{\alpha}(s^{2}+3xs+|x-y|^{2})^{\alpha}(x+s)^{m+1}}.
\]
From the same computations as those in (\ref{eq:limitquadratic}),
we have 
\begin{align*}
\left|\ell(s)^{2}+x\ell(s)+r^{2}\right| & \ge\Re\left(\ell(s)^{2}+x\ell(s)+r^{2}\right)\\
 & =s^{2}+3xs+|x-y|^{2}-\varepsilon^{2}.
\end{align*}
As $x>0$, we conclude that there is a constant (independent of $s$
and $\epsilon$) $M>0$ , such that for all sufficiently small $\varepsilon$
and $s\in(0,\infty)$ 
\[
\frac{1}{\left|\ell(s)^{2}+x\ell(s)+r^{2}\right|^{\alpha}}\le M
\]
from which the inequalities $|x-\ell(s)|\ge s$ and $|\ell(s)|\ge x+s$
give 
\[
|G_{m}(z,\ell(s))|=\frac{1}{|z|^{\alpha}|x-\ell(s)|^{\alpha}|\ell(s)^{2}+x\ell(s)+r^{2}|^{\alpha}|\ell(s)|^{m+1}}\le\frac{M}{|z|^{\alpha}s^{\alpha}(x+s)^{m+1}}
\]
where the last expression is integrable as a function in $s\in(0,\infty)$.
Thus, the Lebesgue's Dominated Convergence Theorem (see Chapter 4
of \cite{R}) implies 
\[
\lim_{\varepsilon\to0}\int_{\gamma}G_{m}(z,t)dt=\int_{0}^{\infty}\frac{e^{i\pi\alpha}}{(-z)^{\alpha}s^{\alpha}(s^{2}+3xs+|x-y|^{2})^{\alpha}(x+s)^{m+1}}ds
\]
and similarly 
\[
\lim_{\varepsilon\to0}\int_{\overline{\gamma}}G_{m}(z,t)dt=\int_{0}^{\infty}\frac{e^{-i\pi\alpha}}{(-z)^{\alpha}s^{\alpha}(s^{2}+3xs+|x-y|^{2})^{\alpha}(x+s)^{m+1}}ds.
\]
Hence, it follows that 
\[
\int_{\gamma}G_{m}(z,t)dt-\int_{\overline{\gamma}}G_{m}(z,t)dt\to\int_{0}^{\infty}\frac{2i\sin(\alpha\pi)}{(-z)^{\alpha}s^{\alpha}(s^{2}+3xs+|x-y|^{2})^{\alpha}(x+s)^{m+1}}dt
\]
as $\varepsilon\to0$. 
\end{proof}
In the next lemma, we handle the vertical rays of $\gamma$ (Figure
\ref{graph}) in a similar manner.

\begin{lemma} \label{vertical} Let $x$, $y$, $r$, and $\theta$
be as in Lemma \ref{rootprops}, $\varepsilon>0$, and define the
paths $\ell_{+},\ell_{-}:(0,\infty)\to\C$ by 
\begin{align*}
\ell_{+}(s) & =y+is+\varepsilon,\\
\ell_{-}(s) & =y+is-\varepsilon.
\end{align*}
Then 
\[
\begin{split} & \int_{\ell_{+}}G_{m}(z,t)dt-\int_{\ell_{-}}G_{m}(z,t)dt\\
 & \to\int_{0}^{\infty}\frac{2\sin(\alpha\pi)}{(-z)^{\alpha}(x-y-it)^{\alpha}(t^{2}+2(r\sin\theta)t)^{\alpha}(y+it)^{m+1}}dt
\end{split}
\]
as $\varepsilon\to0$. \end{lemma} 
\begin{proof}
First of all, since 
\[
y=r\cos\theta+ir\sin\theta,
\]
we have 
\begin{align*}
(\ell_{+}(s)-r\cos\theta)^{2} & =(i(s+r\sin\theta)+\varepsilon)^{2}\\
 & =-s^{2}-2(r\sin\theta)s-(r\sin\theta)^{2}+2i\varepsilon(s+r\sin\theta)+\varepsilon^{2}.
\end{align*}
Consequently 
\begin{align}
\ell_{+}(s)^{2}+x\ell_{+}(s)+r^{2} & =(\ell_{+}(s)-r\cos\theta)^{2}+(r\sin\theta)^{2}\nonumber \\
 & =-(s^{2}+2(r\sin\theta)s-\varepsilon^{2})+2i\varepsilon(s+r\sin\theta).\label{eq:lplus}
\end{align}
Since the real part of the last expression is negative for sufficiently
small $\varepsilon$ and the imaginary part is positive 
\[
\lim_{\varepsilon\to0}\Arg(\ell_{+}(s)^{2}+x\ell_{+}(s)+r^{2})=\pi
\]
from which we deduce that 
\[
\lim_{\varepsilon\to0}(\ell_{+}(s)^{2}+x\ell_{+}(s)+r^{2})^{\alpha}=(s^{2}+2(r\sin\theta)s)^{\alpha}e^{i\alpha\pi},
\]
and consequently 
\[
\lim_{\varepsilon\to0}G_{m}(z,\ell_{+}(s))=\frac{e^{-i\alpha\pi}}{(-z)^{\alpha}(x-y-is)^{\alpha}(s^{2}+2(r\sin\theta)s)^{\alpha}(y+is)^{m+1}}.
\]
From similar computations, we have 
\[
\lim_{\varepsilon\to0}G_{m}(z,\ell_{-}(s))=\frac{e^{i\alpha\pi}}{(-z)^{\alpha}(x-y-is)^{\alpha}(s^{2}+2(r\sin\theta)s)^{\alpha}(y+is)^{m+1}}.
\]
To swap the integral of $G_{m}(z,\ell_{+}(s))$ and the limit as $\varepsilon\rightarrow0$,
we note from (\ref{eq:lplus}) that for $s>0$ and sufficiently small
$\epsilon$ 
\[
\frac{1}{|\ell_{+}(s)^{2}+x\ell_{+}(s)+r^{2}|}\le\frac{1}{s^{2}+2rs\sin\theta-\epsilon^{2}}\le\frac{1}{2rs\sin\theta}.
\]
As $x\ne y+is$ for all $s\in(0,\infty)$, there is a constant (independent
of $s$ and $\varepsilon$) $M>0$ so that for all sufficiently small
$\varepsilon$ 
\[
\frac{1}{|x-\ell_{+}(s)|}=\frac{1}{|x-y-is-\varepsilon|}\le M
\]
from which and 
\[
|\ell_{+}(s)|\ge\sqrt{y^{2}+s^{2}},
\]
we conclude that 
\[
|G_{m}(z,\ell_{+}(s))|\le\frac{M}{(-z)^{\alpha}(2rs\sin\theta)^{\alpha}\sqrt{y^{2}+s^{2}}^{m+1}}.
\]
Since the right side is integrable as a function in $s\in(0,\infty)$
as $0<\alpha<1$, the Dominated Convergence Theorem gives 
\begin{align*}
\lim_{\varepsilon\to0}\int_{\ell_{+}}G_{m}(z,t)dt & =\lim_{\varepsilon\rightarrow0}i\int_{0}^{\infty}G_{m}(z,\ell_{+}(s))ds\\
 & =i\int_{0}^{\infty}\frac{e^{-i\alpha\pi}}{(-z)^{\alpha}(x-y-is)^{\alpha}(s^{2}+2(r\sin\theta)s)^{\alpha}(y+is)^{m+1}}ds.
\end{align*}
Applying similar arguments we have 
\begin{align*}
\lim_{\varepsilon\to0}\int_{\ell_{-}}G_{m}(z,t)dt & =\lim_{\varepsilon\rightarrow0}i\int_{0}^{\infty}G_{m}(z,\ell_{-}(s))ds\\
 & =i\int_{0}^{\infty}\frac{e^{i\alpha\pi}}{(-z)^{\alpha}(x-y-is)^{\alpha}(s^{2}+2(r\sin\theta)s)^{\alpha}(y+is)^{m+1}}ds
\end{align*}
from which the lemma follows. 
\end{proof}
Finally, from Lemmas \ref{int1} through \ref{vertical} we arrive
at the following representation of $P_{m}^{(\alpha)}(z)$ as the difference
of two integrals.

\begin{lemma}
   \label{intrep} Let $x$, $y$, $r$, and $\theta$
be as in Lemma \ref{rootprops}. Define 
\begin{equation}
A_{m}(t,\theta):=\frac{1}{t^{\alpha}(t^{2}+3xt+|x-y|^{2})^{\alpha}(x+t)^{m+1}}\label{eq:Amdef}
\end{equation}
and 
\begin{equation}
B_{m}(t,\theta):=\frac{1}{(x-y-it)^{\alpha}(t^{2}+2(r\sin\theta)t)^{\alpha}(y+it)^{m+1}}.\label{eq:Bmdef}
\end{equation}
Then 
\[
P_{m}^{(\alpha)}(z(\theta))=\frac{\sin(\alpha\pi)}{\pi(-z)^{\alpha}}\left(\int_{0}^{\infty}A_{m}(t,\theta)dt-2\Im{\int_{0}^{\infty}B_{m}(t,\theta))dt}\right).
\]
\end{lemma}

\section{Asymptotic Behavior}

\label{asymp} Now we study the zero distribution of the polynomials
$\Poly{m}(z)$ for large $m$ by examining the behavior of the integrals
$\int_{0}^{\infty}A_{m}(t,\theta)dt$ and $\int_{0}^{\infty}B_{m}(t,\theta)dt$,
defined in Lemma \ref{intrep}. In the next two lemmas, we find
an upper bound for the first integral and a representation for the
second integral similar to that in Watson's lemma.

\begin{lemma} \label{upperbound} Let $x$, $y$, $r$, and $\theta$
be as in Lemma \ref{rootprops}, and let $\alpha\in(0,1)$. The following
inequality holds for all $\theta\in(\frac{2\pi}{3},\pi)$ and $m\geq0$:
\[
\int_{0}^{\infty}A_{m}(t,\theta)dt<\frac{\Gamma(1-\alpha)}{|x-y|^{2\alpha}x^{m+\alpha}m^{1-\alpha}}.
\]
\end{lemma} 
\begin{proof}
The inequality
\begin{equation}
\frac{1}{(t^{2}+3xt+|x-y|^{2})^{\alpha}}\leq\frac{1}{|x-y|^{2 \alpha}}\label{eq:factorAmineq}
\end{equation}
and (\ref{eq:Amdef}) give us the following
bound 
\[
\int_{0}^{\infty}A_{m}(t,\theta)dt\leq\frac{1}{|x-y|^{2\alpha}}\int_{0}^{\infty}\frac{1}{t^{\alpha}(x+t)^{m+1}}dt
\]
for all $\theta\in(\frac{2\pi}{3},\pi)$. We factor out $x^{m+1}$
from the denominator of the integrand and then apply the change of
variables $u=t/x$ to obtain 
\[
\int_{0}^{\infty}A_{m}(t,\theta)dt\leq\frac{1}{|x-y|^{2\alpha}x^{m+\alpha}}\int_{0}^{\infty}\frac{1}{u^{\alpha}(1+u)^{m+1}}du.
\]
By the formula 
\[
\int_{0}^{\infty}\frac{t^{z_{1}-1}}{(1+t)^{z_{1}+z_{2}}}dt=\frac{\Gamma(z_{1})\Gamma(z_{2})}{\Gamma(z_{1}+z_{2})},
\]
(see Chapter 6 of \cite{AS}), we obtain 
\[
\int_{0}^{\infty}A_{m}(t,\theta)dt\leq\frac{1}{|x-y|^{2\alpha}x^{m+\alpha}}\cdot\frac{\Gamma(1-\alpha)\Gamma(m+\alpha)}{\Gamma(m+1)}.
\]
Gautschi's inequality, (see \cite{FQ}) 
\[
m^{1-\alpha}<\frac{\Gamma(m+1)}{\Gamma(m+\alpha)},
\]
gives 
\[
\int_{0}^{\infty}A_{m}(t,\theta)dt<\frac{\Gamma(1-\alpha)}{|x-y|^{2\alpha}x^{m+\alpha}}\cdot\frac{1}{m^{1-\alpha}}.
\]
\end{proof}
\begin{lemma} \label{watsonform} Let $x$, $y$, $r$, and $\theta$
be as in Lemma \ref{rootprops}. For each $\theta\in(\frac{2\pi}{3},\pi)$,
define 
\begin{equation}
g(u,\theta)=\frac{1}{\left(\frac{e^{u}-1}{u}\right)^{\alpha}\left(1+\frac{e^{i\theta}(e^{u}-1)}{2\cos\theta+e^{i\theta}}\right)^{\alpha}\left(1-\frac{ie^{i\theta}(e^{u}-1)}{2\sin\theta}\right)^{\alpha}}.\label{eq:gudef}
\end{equation}
Then for any $m\geq0$, we have 
\begin{equation}
\int_{0}^{\infty}B_{m}(t,\theta)dt=\frac{(-i)^{1-\alpha}}{(x-y)^{\alpha}(2r\sin\theta)^{\alpha}y^{m+\alpha}}\int_{0}^{\infty}g(u,\theta)u^{-\alpha}e^{-mu}du.\label{eq:watsonint}
\end{equation}
\end{lemma} 
\begin{proof}
For arbitrary positive real numbers $\varepsilon<R$ we define the
curves $\ell:[\varepsilon/r,R]\to\C$ and $\sigma_{\beta}:[0,\theta-\frac{\pi}{2}]$,
$\beta\in\{\varepsilon,R\}$, by (see Figure \ref{fig:contourBm})
\begin{align*}
\ell(t) & =-iyt,\text{ and}\\
\sigma_{\beta}(t) & =\beta e^{it}.
\end{align*}
Because the function (in $t$) 
\[
\frac{1}{(x-y-it)^{\alpha}(t^{2}+2(r\sin\theta)t)^{\alpha}(y+it)^{m+1}}
\]
has no zeros in the region enclosed by $\gamma_{R}$, $\sigma_{R}$,
$\ell$, and the real interval $[\varepsilon,R]$, we derive 
\begin{equation}
\int_{\varepsilon}^{R}B_{m}(t,\theta)dt+\int_{\sigma_{R}}B_{m}(t,\theta)dt-\int_{\ell}B_{m}(t,\theta)dt-\int_{\sigma_{\varepsilon}}B_{m}(t,\theta)dt=0.\label{split}
\end{equation}
We will take the limit $\varepsilon\to0$ and show that the integral
along $\sigma_{\varepsilon}$ vanishes, and then we take the limit
$R\to\infty$ and show that the integral along $\sigma_{R}$ vanishes.
Indeed, the second integral in the left side of (\ref{split}) approaches
$0$ as $R\rightarrow\infty$ since $|t|=R$ for all $t$ in the image
of $\sigma_{R}$ and 
\begin{align*}
 & \left|\int_{\sigma_{R}}\frac{dt}{(x-y-it)^{\alpha}(t^{2}+2(r\sin\theta)t)^{\alpha}(y+it)^{m+1}}\right|\\
 & \leq\frac{\frac{\pi}{2}R}{R^{\alpha}|R-|2r\sin\theta||^{\alpha}|R-|x-y||^{\alpha}|R-r|^{m+1}}\rightarrow0
\end{align*}
as $R\to\infty$. For $\sigma_{\varepsilon}$, similar computations
give 
\[
\left|\int_{\sigma_{\varepsilon}}\frac{dt}{(x-y-it)^{\alpha}(t^{2}+2(r\sin\theta)t)^{\alpha}(y+it)^{m+1}}\right|\le\frac{\frac{\pi}{2}\epsilon}{\epsilon^{\alpha}(2r\sin\theta-\epsilon)^{\alpha}(|x-y|-\epsilon)^{\alpha}(r-\epsilon)^{m+1}}\rightarrow0
\]
as $\epsilon\rightarrow0$.

From (\ref{split}), we let $\epsilon\rightarrow0$ and $R\rightarrow0$
to obtain 
\begin{align*}
\int_{0}^{\infty}B_{m}(t,\theta)dt & =-iy\int_{0}^{\infty}B_{m}(-iyt,\theta)dt\\
 & =-iy\int_{0}^{\infty}\frac{dt}{\left(x-y-yt\right)^{\alpha}\left((-iyt)^{2}+(2r\sin\theta)(-iyt)\right)^{\alpha}\left(y+yt\right)^{m+1}}.
\end{align*}
Factor out $-iy$ from the middle term and $y$ from the last term in the denominator and apply the change of variables 
\[
e^{u}=1+t
\]
to obtain 
\[
\begin{split} & \int_{0}^{\infty}B_{m}(t,\theta)dt\\
 & =\frac{(-iy)^{1-\alpha}}{y^{m+1}}\int_{0}^{\infty}\frac{du}{(x-ye^{u})^{\alpha}((2r\sin\theta)(e^{u}-1)-iy(e^{u}-1)^{2}))^{\alpha}e^{(m+1)u}}.
\end{split}
\]
We deduce from $y=re^{i\theta}$ that 
\[
((2r\sin\theta)(e^{u}-1)-iy(e^{u}-1)^{2}))^{\alpha}=u^{\alpha}\left(\frac{e^{u}-1}{u}\right)^{\alpha}(2r\sin\theta)^{\alpha}\left(1-\frac{ie^{i\theta}(e^{u}-1)}{2\sin\theta}\right)^{\alpha}.
\]
Similarly Lemma \ref{rootprops} gives

\begin{align*}
(x-ye^{u})^{\alpha} & =(x-y)^{\alpha}\left(1-\frac{y(e^{u}-1)}{x-y}\right)^{\alpha}\\
 & =(x-y)^{\alpha}\left(1-\frac{e^{i\theta}(e^{u}-1)}{x/r-e^{i\theta}}\right)^{\alpha}\\
 & =(x-y)^{\alpha}\left(1-\frac{e^{i\theta}(e^{u}-1)}{-2\cos\theta-e^{i\theta}}\right)^{\alpha}
\end{align*}
and the lemma follows. 
\end{proof}
\begin{figure}
\begin{centering}
\includegraphics[scale=0.3]{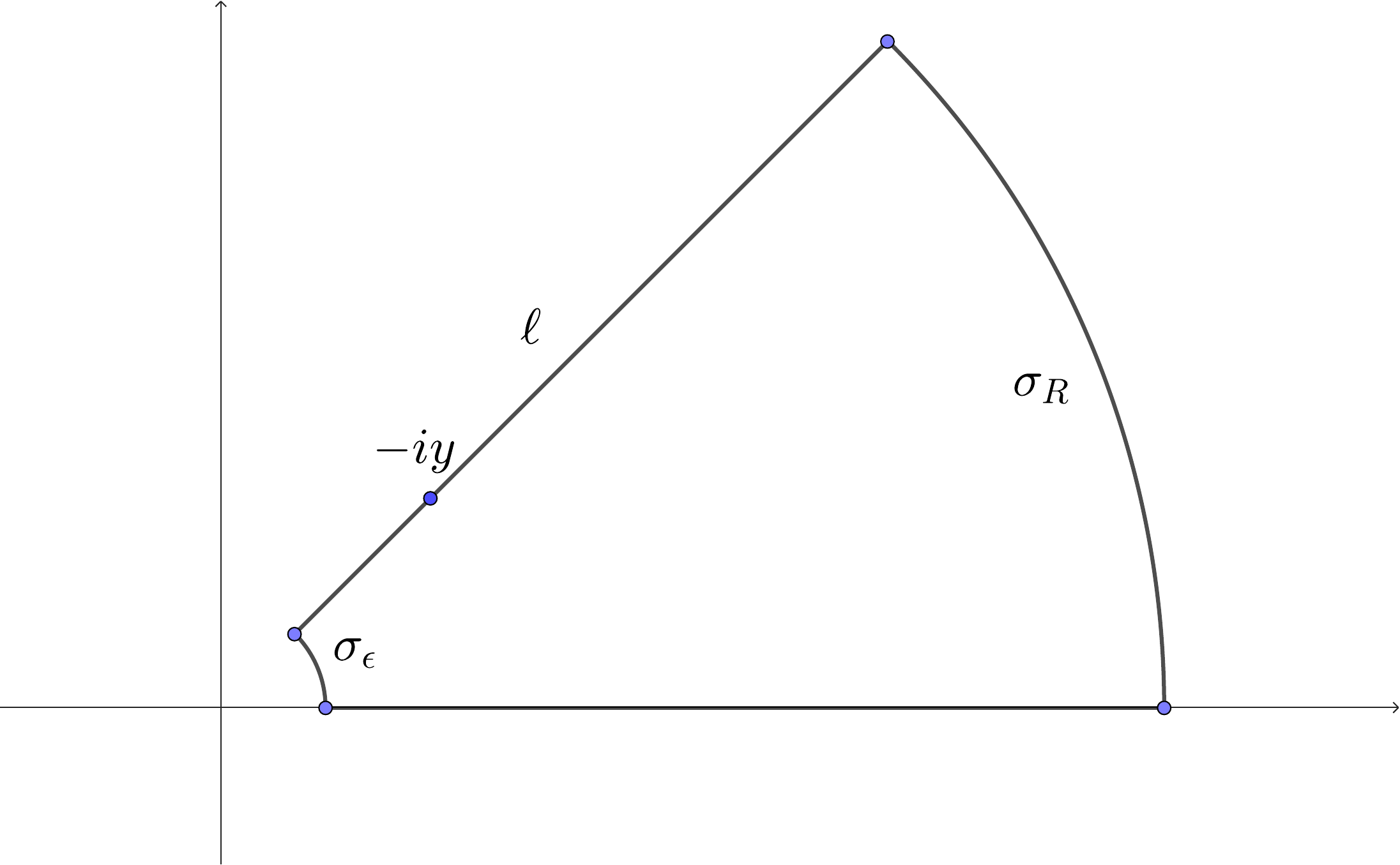} 
\par\end{centering}
\caption{The contour integral of $B_{m}(t,\theta)$}\label{fig:contourBm}
\end{figure}
Our approach to locate the zeros of $P_m^{(\alpha)}(z(\theta))$  given in Lemma \ref{intrep} is to show that the integral of $B_m(t,\theta)$ dominates that of $A_m(t,\theta)$ in modulus. As the integral of $B_m(t,\theta)$ is given in (\ref{eq:watsonint}), we prove this  dominance by finding an asymptotic formula for the integral on the right side of this equation, which is only uniform if $\theta$ stays away from $\pi$. In particular, we have the lemma below. 
\begin{lemma} \label{lem:globalasymp}Define $g(u,\theta)$ as in
(\ref{eq:gudef}). For any small fixed (independent of $m$) $\xi>0$,
we have as $m\rightarrow\infty$, 
\[
\int_{0}^{\infty}g(u,\theta)u^{-\alpha}e^{-mu}du\sim\frac{\Gamma(1-\alpha)}{m^{1-\alpha}}
\]
uniformly on $\theta\in(2\pi/3,\pi-\xi).$ \end{lemma} 
\begin{proof}
We break the range of integration as 
\begin{equation}
\int_{0}^{\infty}g(u,\theta)u^{-\alpha}e^{-mu}du=\int_{0}^{1/\sqrt{m}}g(u,\theta)u^{-\alpha}e^{-mu}du+\int_{1/\sqrt{m}}^{\infty}g(u,\theta)u^{-\alpha}e^{-mu}du.\label{eq:intgsplit}
\end{equation}
We first address the second integral on the right side. Since $g(u,\theta)=\mathcal{O}(1)$
by (\ref{eq:gudef}), we have 
\[
\int_{m^{-1/2}}^{\infty}g(u,\theta)u^{-\alpha}e^{-mu}du=\mathcal{O}\left(\int_{m^{-1/2}}^{\infty}u^{-\alpha}e^{-mu}du\right)\underset{u\rightarrow u/m}{=}\mathcal{O}\left(\frac{\Gamma(1-\alpha,\sqrt{m})}{m^{1-\alpha}}\right)
\]
where $\Gamma(-,-)$ is the upper incomplete gamma function
defined by 
\[
\Gamma(s,x)=\int_{x}^{\infty}z^{s-1}e^{-z}dz.
\]
From the asymptotic formula 
\[
\Gamma(s,x)\sim x^{s-1}e^{-x}
\]
as $x\rightarrow\infty$, we conclude that 
\begin{equation}
\int_{m^{-1/2}}^{\infty}g(u,\theta)u^{-\alpha}e^{-mu}du=\mathcal{O}\left(\frac{e^{-\sqrt{m}}}{m^{1-\alpha/2}}\right).\label{eq:boundtailint}
\end{equation}

We next address the first integral of \eqref{eq:intgsplit}: 
\[
\int_{0}^{1/\sqrt{m}}g(u,\theta)u^{-\alpha}e^{-mu}du.
\]
We note that for big $m$, $u\in(0,1/\sqrt{m})$ is small and 
\[
\left(\frac{e^{u}-1}{u}\right)^{-\alpha}=1+\mathcal{O}(1/\sqrt{m}).
\]
Similarly as $1/(2\cos\theta+e^{i\theta})=\mathcal{O}(1)$, 
\[
\left(1+\frac{e^{i\theta}(e^{u}-1)}{2\cos\theta+e^{i\theta}}\right)^{-\alpha}=1+\mathcal{O}(1/\sqrt{m}).
\]
We conclude that 
\[
\int_{0}^{1/\sqrt{m}}g(u,\theta)u^{-\alpha}e^{-mu}du=\int_{0}^{1/\sqrt{m}}\left(1-\frac{ie^{i\theta}(e^{u}-1)}{2\sin\theta}\right)^{-\alpha}u^{-\alpha}e^{-mu}(1+\mathcal{O}(1/\sqrt{m}))du.
\]

As $\theta\in(2\pi/3,\pi-\xi)$, we have $1/\sin\theta=\mathcal{O}(1)$
and consequently 
\[
\left(1-\frac{ie^{i\theta}(e^{u}-1)}{2\sin\theta}\right)^{-\alpha}=1+\mathcal{O}(1/\sqrt{m}),
\]
from which we conclude that 
\begin{align*}
\int_{0}^{1/\sqrt{m}}g(u,\theta)u^{-\alpha}e^{-mu}du & \sim\int_{0}^{1/\sqrt{m}}u^{-\alpha}e^{-mu}du\\
 & \underbrace{\sim}_{u\rightarrow u/m}\frac{\Gamma(1-\alpha)}{m^{1-\alpha}}.
\end{align*}
We compare the last expression with \eqref{eq:boundtailint} to conclude
the proof of this lemma. 

\end{proof}
With Lemma \ref{lem:globalasymp} at our disposal, we are ready to show that in the formula of $P_m^{(\alpha)}(z(\theta))$ given in Lemma \ref{intrep}, the integral of $B_m(t,\theta)$ dominates (in modulus) that of $A_m(t,\theta)$. We will see in the beginning of the next section that this dominance plays a crucial role in locating the zeros of $P_m^{(\alpha)}(z(\theta))$.
\begin{lemma} \label{inequality} Define $A_{m}(t,\theta)$ and $B_{m}(t,\theta)$
as in Lemma \ref{intrep}. There is $M>0$, independent of $\theta$, so that if $m\ge M$
then 
\[
0<\int_{0}^{\infty}A_{m}(t,\theta)dt<\left|\int_{0}^{\infty}B_{m}(t,\theta)dt\right|
\]
for all $\theta\in(2\pi/3,\pi)$. \end{lemma} 
\begin{proof}
In the case $\theta \in (2\pi/3,\pi - \xi)$, this lemma follows directly from Lemmas \ref{upperbound}, \ref{watsonform}  and  \ref{lem:globalasymp} and 
\[
2\sin\theta<\sqrt{3}<\sqrt{8\cos^{2}\theta+1} = \frac{|x-y|}{r}
\]
where the last equation comes from Lemma \ref{rootprops}. 

It remains to consider $\theta \in [\pi-\xi,\pi)$. In this case, to find a lower bound for
\[
\int_{0}^{\infty}g(u,\theta)u^{-\alpha}e^{-mu}du
\]
we apply similar arguments to those in the proof of the previous lemma and it suffices to find a lower bound for the first integral of \eqref{eq:intgsplit}: 
\[
\int_{0}^{1/\sqrt{m}}g(u,\theta)u^{-\alpha}e^{-mu}du=\int_{0}^{1/\sqrt{m}}\left(1-\frac{ie^{i\theta}(e^{u}-1)}{2\sin\theta}\right)^{-\alpha}u^{-\alpha}e^{-mu}(1+\mathcal{O}(1/\sqrt{m}))du.
\]
We let $\eta:=\pi-\theta\le\xi$ and deduce from $\sin\theta=\sin\eta$ that 
\begin{equation}
\left(1-\frac{ie^{i\theta}(e^{u}-1)}{2\sin\theta}\right)^{-\alpha}(1+\mathcal{O}(u))=\left(1+\frac{i}{2}\frac{u}{\eta}\left(1+\mathcal{O}(u+\eta)\right)\right)^{-\alpha}(1+\mathcal{O}(u)).\label{eq:gusmallu}
\end{equation}
We note that for small $u$ and $\eta$ 
\begin{align*}
 & \left|\sin\left(\Arg\left(\left(1+\frac{i}{2}\frac{u}{\eta}\left(1+\mathcal{O}(u+\eta)\right)\right)^{-\alpha}(1+\mathcal{O}(u))\right)\right)\right|\\
= & \left|\sin\left(-\alpha\Arg\left(1+\frac{i}{2}\frac{u}{\eta}\left(1+\mathcal{O}(u+\eta\right)\right)+\mathcal{O}(u)\right)\right|\\
> & \alpha\sin\left(\Arg\left(1+\frac{i}{2}\frac{u}{\eta}\right)\right)
\end{align*}
since $\alpha\in(0,1)$ and $1+\frac{i}{2}\frac{u}{\eta}\left(1+\mathcal{O}(u+\eta\right)$ lies in the first quadrant for small $u$ and $\eta$. The last expression
is 
\[
\frac{\alpha u/2\eta}{\sqrt{1+(u/2\eta)^{2}}}\ge\begin{cases}
\frac{\alpha u}{2\sqrt{2}\eta} & \text{ if }u/2\eta\le1\\
\frac{\alpha u/2\eta}{\sqrt{2(u/2\eta)^{2}}}=\frac{\alpha}{\sqrt{2}} & \text{ if }u/2\eta>1.
\end{cases}
\]
Note that or small $u$ and $\eta$, the modulus of the main term
of (\ref{eq:gusmallu}) is 
\[
\left(1+\frac{u^{2}}{4\eta^{2}}\right)^{-\alpha/2}\ge\begin{cases}
2^{-\alpha/2} & \text{ if }u/2\eta\le1\\
\left(\frac{u}{\sqrt{2}\eta}\right)^{-\alpha} & \text{\text{ if }}u/2\eta>1
\end{cases}.
\]
Thus 
\begin{equation}
\left|\Im\left(\left(1+\frac{i}{2}\frac{u}{\eta}\left(1+\mathcal{O}(u+\eta\right)\right)^{-\alpha}(1+\mathcal{O}(u))\right)\right|\ge\begin{cases}
\frac{\alpha u}{2^{(\alpha+3)/2}\eta} & \text{ if }u/2\eta\le1\\
\frac{\alpha}{\sqrt{2}}(\frac{u}{\sqrt{2}\eta})^{-\alpha} & \text{ if }u/2\eta>1
\end{cases}.\label{eq:Reglowerbound}
\end{equation}
We note that 
\begin{equation} \label{eq:lowerboundintg}
\left|\int_{0}^{1/\sqrt{m}}g(u,\theta)u^{-\alpha}e^{-mu}du\right|\ge\int_{0}^{1/\sqrt{m}}\left|\Im\left(\left(1-\frac{ie^{i\theta}(e^{u}-1)}{2\sin\theta}\right)^{-\alpha}(1+\mathcal{O}(u))\right)\right|u^{-\alpha}e^{-mu}du.
\end{equation}
If $2\eta>1/\sqrt{m}$, then the right side is 
\begin{align*}
&> \frac{\alpha}{\eta 2^{(\alpha+3)/2}}\int_0^{1/\sqrt{m}}u^{-\alpha+1}e^{-mu}du   \\
&= \frac{\alpha}{\eta 2^{(\alpha+3)/2}m^{2-\alpha}}\int_0^{\sqrt{m}}u^{-\alpha+1}e^{-u}du \\
&= \frac{\alpha \gamma(-\alpha+2,\sqrt{m})}{\eta 2^{(\alpha+3)/2}m^{2-\alpha}}
\end{align*}
where $\gamma(-,-)$ is the lower incomplete gamma function defined by 
\[
\gamma(s,x)=\int_0^x z^{s-1}e^{-z}dz.
\]
For large $m$, the inequality
\[
\int_{0}^{\infty}A_{m}(t,\theta)dt<\left|\int_{0}^{\infty}B_{m}(t,\theta)dt\right|
\]
follows from Lemmas \ref{watsonform}  and \ref{upperbound} and the fact that  for small $\xi$,  we have $|x-y|^2=(18+9/4)(1+\mathcal{O}(\eta))$ and  $|y|=r<2x/3$ (see Lemma \ref{rootprops}).

On the other hand, if $2\eta \le 1/\sqrt{m}$, then from (\ref{eq:Reglowerbound}) the right side of (\ref{eq:lowerboundintg}) is at least
\[
\frac{\alpha}{\eta 2^{(\alpha+3)/2}} \int_0^{2\eta} u^{-\alpha+1}e^{-mu}du+\frac{\alpha \eta^\alpha}{2^{(\alpha+1)/2}}\int_{2\eta}^{1/\sqrt{m}} u^{-2\alpha}e^{-mu}du
\]
which, after the substitution $u \rightarrow u/m$, becomes
\begin{equation} \label{eq:sumtwoints}
  \frac{\alpha}{\eta 2^{(\alpha+3)/2}m^{2-\alpha}}\int_0^{2\eta m }u^{-\alpha+1}e^{-u}du+\frac{\alpha \eta^\alpha }{2^{(\alpha+1)/2}m^{1-2\alpha}}\int_{2\eta m}^{\sqrt{m}} u^{-2\alpha}e^{-u}du.
\end{equation}

If $\eta m=o(1)$, then 
\[
\left| \int_{0}^{\infty}g(u,\theta)u^{-\alpha}e^{-mu}du\right|
\]
is at least the second summand in the expression above. Thus the fact that $\sin(\theta)=\sin(\eta)=\eta(1+\mathcal{O}(\eta))$ and Lemma \ref{watsonform} give
\[
\left|\int_{0}^{\infty}B_{m}(t,\theta)dt\right|\ge \frac{\alpha}{2^{(\alpha+1)/2}m^{1-2\alpha}(2r)^{\alpha}|y|^{m+\alpha}}\int_{2\eta m}^{\sqrt{m}}u^{-2\alpha}e^{-u}du.
\]
The conclusion of this lemma follows from Lemma \ref{upperbound}, $|x-y|^2=(18+9/4)(1+\mathcal{O}(\eta))$ and  $|y|=r<2x/3$. 

In the final case when $\eta m =\Omega(1)$, 
\[
\left| \int_{0}^{\infty}g(u,\theta)u^{-\alpha}e^{-mu}du\right|
\]
is at least the first summand of  (\ref{eq:sumtwoints}). As the integral in this summand does not approach $0$ as $m\rightarrow \infty$, the conclusion of this lemma follows from the same reasons as those in the case $2\eta>1/\sqrt{m}$. 

\end{proof}

\section{Main Results}

\label{main} In this section, we prove our Theorem \ref{thm:zerodensity}. To prove the first claim of this theorem that for large $m$ all the zeros of $P_m^{(\alpha)}(z)$ lie on $(-\infty,-4/27)$, we will show that number of zeros of $P_{m}^{(\alpha)}(z(\theta))$ on the interval $\theta\in(2\pi/3,\pi)$ is at least $\left\lfloor m/3\right\rfloor $, which is the degree of $P_{m}^{(\alpha)}(z)$. As a consequence of the Fundamental Theorem of Algebra, all the zeros of $P_{m}(z)$ will lie on $z((2\pi/3,\pi))=(-\infty,-4/27)$ (c.f. Lemma \ref{lem:zmonotone}). From Lemmas \ref{intrep} and \ref{inequality}, we conclude
that at a value of $\theta\in(2\pi/3,\pi)$ where 
\[
h_{m}(\theta):=\int_{0}^{\infty}B_{m}(t,\theta)dt
\]
is in $i\mathbb{R}^{+}$ or $i\mathbb{R}^{-}$, $P_{m}^{(\alpha)}(z(\theta))$ is negative or positive respectively. Thus if $\theta_1 < \theta_2 <\ldots <\theta_K\in (2\pi/3,\pi)$ satisfy $h_m(\theta_k)\in i\mathbb{R^+}$ (or $i\mathbb{R^-}$ ) and $h_m(\theta_{k+1})\in i \mathbb{R^-}$ (or $i\mathbb{R^+})$  then by the Intermediate Value Theorem, there is at least a zero of $P_m^{(\alpha)}(z(\theta))$   on the interval $(\theta_k,\theta_{k+1})$, $1\le k<K$.  The number of such intervals could be computed by studying how the curve $h_m(\theta)$ winds around the origin.  In particular, this number is at least
\begin{equation}
\left\lfloor \frac{|\Delta\arg_{(2\pi/3,\theta_K)} h_{m}(\theta)|}{\pi}\right\rfloor =\left\lfloor   \frac{|\Delta\arg_{(2\pi/3,\pi)} h_{m}(\theta) - \Delta\arg_{(\theta_K,\pi)} h_{m}(\theta)|}{\pi}  \right\rfloor.\label{eq:numberzeros}
\end{equation}
By (\ref{eq:watsonint}), the change in argument of $h_{m}(\theta)$
on $\theta\in(2\pi/3,\pi)$ is 
\begin{equation}
-\alpha\Delta\arg_{(2\pi/3,\pi)}(x-y)-(m+\alpha)\Delta\arg_{(2\pi/3,\pi)} y+\Delta\arg_{(2\pi/3,\pi)} \int_{0}^{\infty}g(u,\theta)u^{-\alpha}e^{-mu}du.\label{eq:changearg}
\end{equation}
Since  $\lim_{\theta\rightarrow\pi }\Arg(x-y)=0$ (as $y=r e^{i\theta}$ and $x\in\mathbb{R^+}$) and  (see Figure \ref{fig:anglerelation})
\[
\lim_{\theta\rightarrow 2\pi/3}-\Arg(x-y)=\lim_{\theta \rightarrow 2\pi/3}\tan^{-1}\left(\frac{1}{3}\tan(\pi-\theta)\right) = \frac{\pi}{6},
\]
we have  
\[
-\alpha\Delta\arg_{(2\pi/3,\pi)}(x-y)= -\frac{\alpha \pi}{6}.
\]
From $y=r e^{i\theta}$, the second term of (\ref{eq:changearg}) is 
\[
-\frac{(m+\alpha)\pi}{3}.
\]
Finally, Lemma \ref{lem:globalasymp} yields
\[
\Delta\arg_{(2\pi/3,\pi)} \int_{0}^{\infty}g(u,\theta)u^{-\alpha}e^{-mu}du=\lim_{\theta\rightarrow \pi}\Arg\int_0^\infty g(u,\theta)u^{-\alpha}e^{-mu}du
\]
where the right side lie on $(-\pi/2,0)$ by  (\ref{eq:gusmallu}). 
Hence
\begin{equation} \label{eq:changeargwholeint}
\Delta\arg_{(2\pi/3,\pi)} h_{m}(\theta) = -\frac{m\pi}{3}   - \frac{\alpha \pi}{2}+\lim_{\theta \rightarrow \pi}\Arg\int_0^\infty g(u,\theta)u^{-\alpha}e^{-mu}du. 
\end{equation}

To compute 
\[
\Delta\arg_{(\theta_K,\pi)} h_{m}(\theta)
\]
we note that as $\theta\rightarrow \pi$, $x-y\rightarrow 3+3/2$ (see Lemma \ref{rootprops})  and thus
\begin{align*}
&\lim_{\theta\rightarrow\pi}\sin\left(\Arg\left(\int_{0}^{\infty}B_{m}(t,\theta)dt\right)\right) \\
 = &\sin\left(-\frac{\pi(1-\alpha)}{2}-(m+\alpha)\pi+\lim_{\theta \rightarrow \pi}\Arg\int_0^\infty g(u,\theta)u^{-\alpha}e^{-mu}du\right) \\
  =& (-1)^m \sin\left(-\frac{\pi}{2}-\frac{\alpha\pi}{2}+\lim_{\theta \rightarrow \pi}\Arg\int_0^\infty g(u,\theta)u^{-\alpha}e^{-mu}du \right) .
\end{align*}
From the equation above and the fact that  $0<\alpha<1$ and 
\[
\Arg\int_0^\infty g(u,\theta)u^{-\alpha}e^{-mu}du  \in(-\pi/2,0),
\]
we conclude if $\theta_K\in(2\pi/3,\pi)$ is the largest angle so that $h_m(\theta_K)\in i\mathbb{R}^{\pm}$ or equivalently 
\[
\sin\left(\Arg\left(\int_{0}^{\infty}B_{m}(t,\theta_K)dt\right)\right)=\pm 1,
\]
then 
\[
\Delta\arg_{(\theta_K,\pi)} h_{m}(\theta) = -\frac{\alpha \pi}{2}+\lim_{\theta \rightarrow \pi}\Arg\int_0^\infty g(u,\theta)u^{-\alpha}e^{-mu}du.
\]
Thus by (\ref{eq:changeargwholeint}),  (\ref{eq:numberzeros}) is at least 
\[
\left \lfloor \frac{m}{3} \right \rfloor .
\]
 
\begin{figure}
\begin{centering}
\includegraphics[scale=0.3]{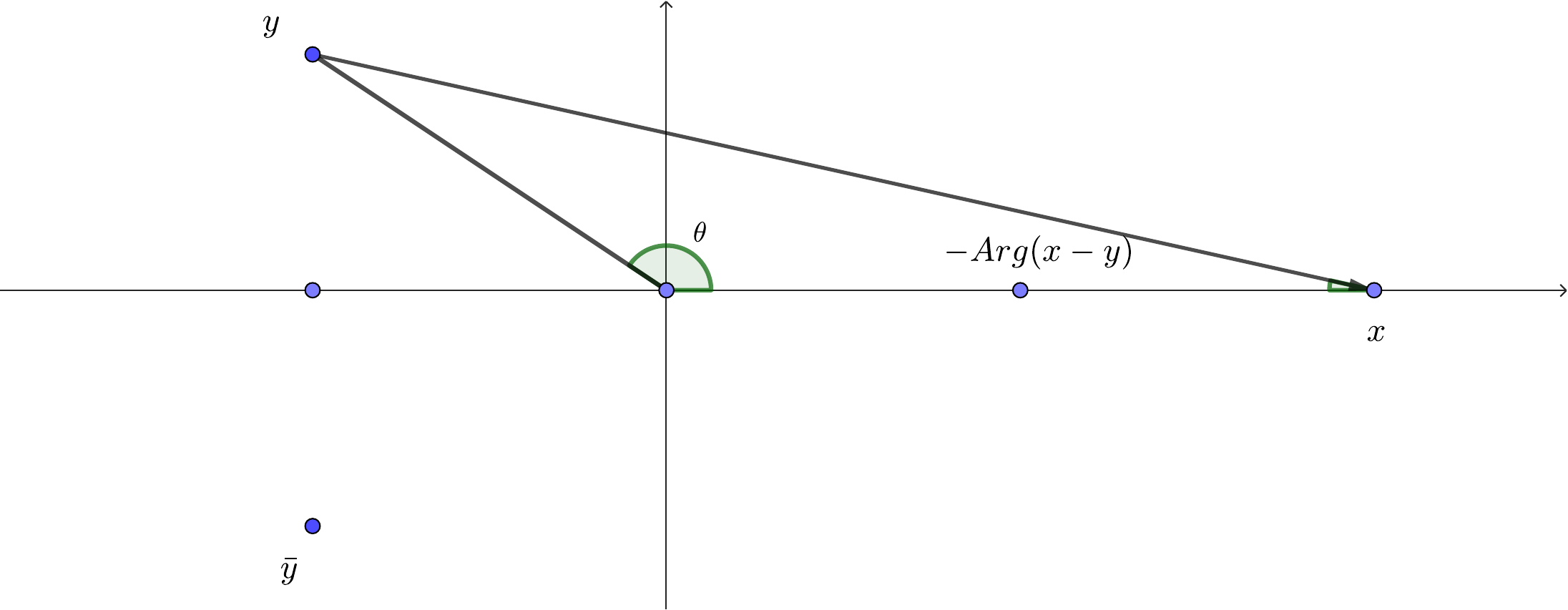}
\par\end{centering}
\caption{$\Arg(x-y)$}\label{fig:anglerelation}
\end{figure}
Each zero of $P_m ^{(\alpha)}(z(\theta))$ on $(2\pi/3,\pi)$  produces one zero of $P_m^{\alpha}(z)$ on the interval $(-\infty,-4/27)$ by the lemma below.
\begin{lemma}
\label{lem:zmonotone}The function 
\begin{equation}
z(\theta)=\frac{4\cos^{2}\theta}{(1-4\cos^{2}\theta)^{3}}\label{eq:ztheta}
\end{equation}
is increasing on $(2\pi/3,\pi)$ and it maps this interval to $(-\infty,-4/27)$. 
\end{lemma}

\begin{proof}
This lemma follows from
\[
z'(\theta)=-\frac{8\sin\theta\cos\theta\left(8\cos^{2}\theta+1\right)}{(1-4\cos^{2}\theta)^{4}}>0
\]
and
\begin{align*}
\lim_{\theta\rightarrow2\pi/3}\frac{4\cos^{2}\theta}{(1-4\cos^{2}\theta)^{3}} & =-\infty,\\
\lim_{\theta\rightarrow\pi}\frac{4\cos^{2}\theta}{(1-4\cos^{2}\theta)^{3}} & =-\frac{4}{27}.
\end{align*}
\end{proof}

As the degree of $P_m^{\alpha}(z)$ is  $\lfloor m/3 \rfloor$, all the zeros of $P_m^{(\alpha)}(z)$ must lie on $(-\infty,-4/27)$.  For the remainder of the paper, we will find the the limiting probability density function of the zeros of $P_{m}^{(\alpha)}(z)$ as $m\rightarrow \infty$. For each $m>0$ , $z\in(-\infty-4/27)$, and small $\varepsilon>0$, let $N_{m,\epsilon}(z)$ be the number of zeros of $P_{m}^{(\alpha)}(z)$
on the interval $(z,z+\varepsilon)$. The limiting probability density function of the zeros of $P_{m}^{(\alpha)}(z)$ evaluated at each
$z\in(-\infty,-4/26)$ is defined as 
\[
\lim_{\varepsilon\rightarrow0}\frac{1}{\varepsilon}\lim_{m\rightarrow\infty}\frac{N_{m,\epsilon}(z)}{\left\lfloor m/3\right\rfloor }.
\]
Since the function $z(\theta)$ defined in (\ref{eq:ztheta}) is an increasing map, $N_{m,\varepsilon}(z)$ is the same as the number
of zeros of $P_{m}^{(\alpha)}(z(\theta))$ on the interval $(W(z),W(z+\epsilon))$, where $W(z)$ is the inverse of the map $z(\theta)$. From similar arguments to those in beginning of this section, this number is 
\[
\frac{1}{\pi}|\Delta \arg_{(W(z),W(z+\epsilon))}h_m(\theta)|+\mathcal{O}(1).
\]
By Lemma $\ref{watsonform}$, this expression is
\[
m\Delta\arg_{(W(z),W(z+\epsilon))} y+\mathcal{O}(1)
\]
and consequently as $y=re^{i\theta}$,
\[
N_{m,\varepsilon}(z)=\frac{m}{\pi}\left(W(z+\epsilon)-W(z)\right)+\mathcal{\mathcal{O}}(1).
\]
Thus 
\begin{equation}
\lim_{\varepsilon\rightarrow0}\frac{1}{\varepsilon}\lim_{m\rightarrow\infty}\frac{N_{m,\epsilon}(z)}{\left\lfloor m/3\right\rfloor }=\frac{3}{\pi}\lim_{\varepsilon\rightarrow0}\frac{W(z+\epsilon)-W(z)}{\varepsilon}=\frac{3}{\pi}W'(z).\label{eq:limitprob}
\end{equation}
Since 
\[
\frac{d}{d\theta}\left(\frac{4\cos^{2}\theta}{(1-4\cos^{2}\theta)^{3}}\right)=-\frac{4\sin(2\theta)(8\cos^{2}\theta+1)}{(1-4\cos^{2}\theta)^{4}},
\]
the last expression of (\ref{eq:limitprob}) becomes 
\begin{equation}
-\frac{3(1-4\cos^{2}\theta)^{4}}{4\pi\sin(2\theta)(8\cos^{2}\theta+1)}\label{eq:limitprobtheta}
\end{equation}
where $\theta=W(z)$.
We recall from Lemma \ref{rootprops} that the real root $x$ of $1+t+zt^{3}$
satisfies 
\[
4\cos^{2}\theta-1=x
\]
and hence as $\theta\in(2\pi/3,\pi)$ 
\[
\sin(2\theta)=-\sqrt{1-\cos^{2}(2\theta)}=-\sqrt{1-\left(2\cos^{2}\theta-1\right)^{2}}=-\frac{\sqrt{4-(x-1)^{2}}}{2}.
\]
Thus (\ref{eq:limitprobtheta}) becomes 
\[
\frac{3x^{4}}{2\pi(3+2x)\sqrt{4-(x-1)^{2}}}.
\]
Using 
\[
x^{3}=-\frac{x+1}{z},
\]
which comes from the fact that $x$ is a zero of $1+t+zt^{3}$, we
rewrite this expression as 
\[
-\frac{3x(x+1)}{2\pi z(3+2x)\sqrt{4-(1-x)^{2}}}=-\frac{3x\sqrt{x+1}}{2\pi z(3+2x)\sqrt{3-x}}
\]
and Theorem \ref{thm:zerodensity} follows.

\bibliographystyle{plain}
\bibliography{refs}

\end{document}